\documentclass{article}

\usepackage{amsmath,amscd}
\usepackage{amssymb,amsthm}
\usepackage{enumerate}

\newtheorem{thm}{Theorem}[section]

\newtheorem{proposition}[thm]{Proposition}
\newtheorem{corollary}[thm]{Corollary}
\newtheorem{conjecture}[thm]{Conjecture}
\newtheorem{question}[thm]{Question}

\theoremstyle{definition}
\newtheorem{definition}[thm]{Definition}
\newtheorem{remark}[thm]{Remark}

\newcommand{\pr}{\mathbb{P}}
\newcommand{\Z}{\mathbb{Z}}
\newcommand{\Q}{\mathbb{Q}}
\newcommand{\R}{\mathbb{R}}
\newcommand{\C}{\mathbb{C}}
\newcommand{\NC}{\operatorname{N}_1}
\newcommand{\ND}{\operatorname{N}^1}
\newcommand{\NE}{\operatorname{NE}}
\newcommand{\Rat}{{\operatorname{RatCurves}}^{n}}
\newcommand{\Bl}{\operatorname{Bl}}
\newcommand{\Exc}{\operatorname{Exc}}
\newcommand{\Lo}{\operatorname{Locus}}
\newcommand{\Chow}{\operatorname{Chow}}
\newcommand{\Fam}{\operatorname{Fam}}
\newcommand{\Pic}{\operatorname{Pic}}
\newcommand{\sC}{\mathcal{C}}
\newcommand{\sO}{\mathcal{O}}
\newcommand{\sN}{\mathcal{N}}
\newcommand{\sm}{\operatorname{sm}}
\newcommand{\codim}{\operatorname{codim}}
\newcommand{\cont}{\operatorname{cont}}

\title{On a generalization of the Mukai conjecture for Fano fourfolds}
\author{Kento Fujita}

\begin{document}
\maketitle

\begin{abstract}
{\noindent 
Let $X$ be a complex $n$-dimensional Fano manifold.
Let $s(X)$ be the sum of $l(R)-1$ for all the 
extremal rays $R$ of $X$, the edges of the cone $\NE(X)$
of curves of $X$, where $l(R)$ denotes the minimum
of $(-K_X \cdot C)$ for all rational curves $C$ whose
classes $[C]$ belong to $R$. 
We show that $s(X)\leq n$ if $n\leq 4$.
And for $n\leq 4$, we completely classify the case
the equality holds. This is a refinement of the Mukai conjecture 
on Fano fourfolds.} 
\end{abstract}

\section{Introduction}

Let $X$ be an arbitrary $n$-dimensional Fano manifold with 
the Picard number ${\rho}_X$. 
In 1988, Mukai \cite{mukai} made the following conjecture.

\begin{conjecture}\label{Mconj} One has
\[
{\rho}_X(r_X-1)\leq n,
\]
and the equality holds if and only if $X\simeq(\pr^{r_X-1})^{\rho_X}$, where 
\[
r_X:=\max\{m\in\Z_{>0}\mid-K_X\sim mL\text{ for some Cartier divisor }L\}.
\]
\end{conjecture}

There are several approaches and refinements of Conjecture \ref{Mconj}. 
See for example \cite{ACO, gm, toric, NoOc, wisn2}. 
Nowadays, the following conjecture due to Tsukioka \cite{tsu3} (cf.\ \cite{tsuproc}) 
is the most 
generalized version of Conjecture \ref{Mconj}. 

\begin{conjecture}\label{mGMconj} One has
\[
{\rho}_X(l_X-1)\leq n,
\]
and the equality holds if and only if $X\simeq(\pr^{l_X-1})^{\rho_X}$, where
$l_X$ denotes the minimum of the length $l(R)$ 
of all the extremal rays $R$ of $X$,
and 
\[
l(R):=\min\{(-K_X\cdot C)\mid C \subset X
\text{ is a rational curve with }[C]\in R\}.
\]
\end{conjecture}

We think that it is more natural to consider \emph{all} the
extremal rays to study a Fano manifold since 
each extremal ray has various geometric information. 
We set up the following question.

\begin{question}\label{sum_conj}
Give a bound of 
\[
s(X):=\sum_{R\subset\NE(X)\text{ extremal ray}} (l(R)-1)
\]
for arbitrary $n$-dimensional Fano manifolds $X$.
\end{question}

This question is a refinement of Conjectures \ref{Mconj} and \ref{mGMconj} since 
the invariant $s(X)$ satisfies the inequality 
$
\rho_X(r_X-1)\leq\rho_X(l_X-1)\leq s(X).
$
We note that the invariant $s(X)$ is a natural invariant. For example, 
let $X:=\prod_{i=1}^m\pr^{d_i}$ with $\sum_{i=1}^m d_i=n$. Then $s(X)=n$ holds 
despite $\rho_X(l_X-1)=m\cdot\min\{d_i\}$ is less than $n$ unless 
$d_1=\cdots=d_m$. 

In this paper, we identify the bound of $s(X)$ when $n\leq 4$.

\begin{thm}[Main Theorem]\label{mainthm}
Let $X$ be an $n$-dimensional Fano manifold.
\begin{enumerate}
\renewcommand{\theenumi}{\roman{enumi}}
\renewcommand{\labelenumi}{\rm{(\theenumi)}}
\item\label{mainthm1}
If $n\leq 3$, then $s(X)\leq n$ holds. Moreover, the equality holds if and only if 
\[
X\simeq\prod_{R\subset\NE(X)\text{ extremal ray}}{\pr}^{l(R)-1}.
\]
\item\label{mainthm2}
If $n =4$, then $s(X)\leq n$ holds. Moreover, the equality holds if and only if 
\[
X\simeq\prod_{R\subset\NE(X)\text{ extremal ray}}{\pr}^{l(R)-1}
\]
or
\[
X\simeq{\Bl}_{p,q}({\Q}^{4}),
\]
the blowing up of $\Q^4$ along $p$ and $q$, 
where ${\Q}^{4}\subset{\pr}^{5}$ is a smooth hyperquadric and $p$, $q$ 
are distinct points in $\Q^4$ with $\overline{pq}\not\subset{\Q}^{4}$, 
where $\overline{pq}\subset{\pr}^{5}$ is the line through $p$ and $q$.
\end{enumerate}
\end{thm}

\begin{remark}\label{highdim}
If $n\geq 5$, then there exsits an $n$-dimensional Fano manifold $X$ such that $s(X)$ 
is strictly larger than $n$ 
(see Remark \ref{calBCW} \eqref{calBCW3}). However, such $X$ is very special 
as far as we know. We think that all such $X$ should be classified. 
\end{remark}

As an immediate consequence of Theorem \ref{mainthm}, 
we can give the affirmative answer to Conjecture \ref{mGMconj} 
in the case $n\leq 4$. (Tsukioka \cite{tsu3} proved 
the inequality in the case $n=4$
but did not settle the assertion on the equality case.)

\begin{corollary}[cf. \cite{tsu3}]\label{mGM4} 
Conjecture \ref{mGMconj} is true if $n\leq 4$.
\end{corollary}

\smallskip

\noindent\textbf{Acknowledgements.}
The author would like to express his gratitude to Professor Shigefumi Mori 
for warm encouragements and valuable comments.
He also thanks Professors Shigeru Mukai, Noboru Nakayama, Masayuki Kawakita and Stefan Helmke for valuable comments during the seminars in RIMS. 
The author thank the referee for useful comments. 
Especially, the proof of Theorem \ref{mainthm} \eqref{mainthm2} 
has been improved by referee's comments. 
The author is partially supported by a JSPS Fellowship for Young Scientists. 
This paper is a modified version of the author's master thesis 
submitted in January 2011, to RIMS, Kyoto University. 

\smallskip

\noindent\textbf{Notation and terminology.}
We always work over the complex number field $\C$. 
For a proper variety $X$, let $\NC(X)_\Q$ (rep.\ $\ND(X)_\Q$) be the 
vector space of one-cycles (resp.\ Cartier divisors) on $X$, with rational coefficients, 
modulo numerical equivalence. Let $\NC(X):=\NC(X)_\Q\otimes_\Q\R$ and 
$\ND(X):=\ND(X)_\Q\otimes_\Q\R$. The \emph{Picard number} of $X$, 
denoted by $\rho_X$, is defined to be the dimension of the vector space $\NC(X)$.

For an $n$-dimensional 
normal projective variety $X$, we denote the normalization of the 
parameterizing space of 
irreducible and reduced rational curves on $X$ by $\Rat(X)$ 
(see \cite[Definition II.2.11]{kollar}). 
For the theory of extremal contraction, we refer the readers to \cite{KoMo}. 
A projective surjective morphism $f\colon X\rightarrow Z$ is called a 
\emph{contraction morphism} if $Z$ is normal projective and 
any fiber of $f$ is connected.
For an extremal ray
$R\subset\overline{\NE}(X)$, 
we say that $R$ \emph{defines the contraction morphism 
$\cont_R\colon X\rightarrow Y$} if $\cont_R$ is a contraction morphism and 
the kernel of the surjection $\NC(X)\rightarrow\NC(Y)$ is equal to $\R R(=R+(-R))$.
The morphism $\cont_R$ is called 
the \emph{associated contraction morphism}. For example, if $X$ is smooth and $R$ is 
$K_X$-negative, then $R$ defines the contraction morphism.
For an extremal ray
$R\subset\overline{\NE}(X)$, 
we say that $R$ is \emph{of fiber type} (resp.\ \emph{divisorial, small}) if 
$R$ defines the contraction morphism $\cont_R\colon X\rightarrow Y$ and the 
morphism
is of fiber type (resp.\ divisorial, small).
We define 
\[
\Exc(R):=\{ x \in X|\,\,{\cont}_R\colon X\rightarrow Y\text{ is not isomorphism at }x\}.
\]
For example, if $R$ is of fiber type, then $\Exc(R)=X$.
We say that $R$ is \emph{of type} $(a$, $b)$ 
if 
$\dim(\Exc(R))=a$ and $\dim({\cont}_R(\Exc(R)))=b$, 
and we say that $R$ is \emph{of type} ${(n-1, b)}^{\sm}$ if 
the associated contraction morphism is the blowing up morphism of 
a smooth projective variety 
along a smooth subvariety of dimension $b$ (in particular, $X$ must be smooth). 
For an extremal ray $R\subset\overline{\NE}(X)$ and a Cartier divisor $E$ on $X$, 
the notation $(E\cdot R)>0$ (resp.\ $(E\cdot R)<0$, $(E\cdot R)=0$)
means that the property $(E\cdot C)>0$ (resp.\ $(E\cdot C)<0$, $(E\cdot C)=0$)
holds for a curve $C\subset X$ with $[C]\in R$.

For an algebraic variety $X$ and a closed subscheme $Y\subset X$, the morphism 
${\Bl}_Y(X)\rightarrow X$ denotes the blowing up of $X$ along $Y$.
The symbol ${\Q}^n$ denotes a smooth hyperquadric in ${\pr}^{n+1}$.
We say that $X$ is a \emph{Fano manifold} if $X$ is a smooth projective variety 
such that the anticanonical divisor $-K_X$ is ample.

\section{Preliminaries}

\subsection{A family of rational curves}


We observe the definition and a property of a family of rational curves for a fixed 
normal projective variety. 

\begin{definition}[see for example \cite{ACO}]\label{fam_rat}
Let $X$ be a normal projective variety. 
We define a \emph{family of rational curves} 
to be an irreducible component $H\subset\Rat(X)$. 
For any $x\in X$, let $H_x$ be the subvariety of $H$ 
parameterizing rational curves passing through $x$,
and $\tilde{H_x}$ the normalization of the image of $H_x$ in the Chow variety 
$\Chow(X)$.
We define $\Lo(H)$ (resp.\ $\Lo(H_x)$) to be the union of rational curves 
parameterized by $H$ (resp.\ $H_x$). 
For a family $H$ of rational curves on $X$,
the family $H$ is said to be \emph{dominating} if the closure $\overline{\Lo(H)}$ 
is equal to $X$, 
\emph{unsplit} if $H$ is projective, 
and \emph{locally unsplit} if $H_x$ is projective for general $x\in\Lo(H)$.
\end{definition}

The following proposition may be familiar. 

\begin{proposition}[{\cite[Proposition 2.5(b)]{NoOc}}]\label{famineq}
Let $X$ be a smooth projective variety, $H$ be a family of rational curves on $X$, 
and $x\in\Lo(H)$ be a point such that $H_x$ is projective. Then one has 
\[
\dim\Lo(H)+\dim\Lo(H_x)\geq \dim X+(-K_X\cdot\Fam H)-1,
\]
where $\Fam H$ is the numerical class of the curves in $X$ parametrized by $H$. 
\end{proposition}

\subsection{Properties of extremal contractions}

We show some properties of extremal contractions associate to extremal rays that 
we need to prove Theorem \ref{mainthm}. 

\begin{proposition}\label{divptprop}
Let $X$ be an $n$-dimensional smooth projective variety.
Assume that there exist distinct $K_X$-negative extremal rays 
$R_1,R_2\subset\overline{\NE}(X)$ such that 
$R_1$ is of type $(n-1, 0)$,
$l(R_2)\geq 2$ 
and 
$\Exc(R_1)\cap\Exc(R_2)\ne\emptyset$.
Then $R_2$ is of fiber type and ${\rho}_X=2$. 
\end{proposition}

\begin{proof}
Let $E_i:=\Exc(R_i)$ for $i=1$, $2$ and fix $x\in E_1\cap E_2$. 
Let $C\subset X$ be a rational curve such that 
\begin{enumerate}
\renewcommand{\theenumi}{\arabic{enumi}}
\renewcommand{\labelenumi}{(\theenumi)}
\item\label{divptproof1}
$x\in C$ and $[C]\in R_2$,
\item\label{divptproof2}
$(-K_X\cdot C)$ is minimal among satisfying \eqref{divptproof1}.
\end{enumerate}
Let $H$ be a family of rational curves containing $[C]\in\Rat(X)$. 
Then $H_x$ is projective by construction. 
If there exists an irreducible curve $l \subset E_1\cap\Lo(H_x)$ then 
$[l]\in R_1\cap R_2=\{0\}$, which leads to a contradiction. 
Hence $\dim(E_1\cap\Lo(H_x))=0$. 
Thus $\dim\Lo(H_x)\leq 1$ since $\dim E_1=n-1$. Therefore,
\begin{eqnarray*}
1 & \geq & \dim\Lo(H_x)\geq(n-\dim\Lo(H))+(-K_X\cdot\Fam H)-1\\
& \geq & l(R_2)-1\geq 1
\end{eqnarray*}
by Proposition \ref{famineq}. 
Thus $\dim\Lo (H)=n$ and $l(R_2)=(-K_X\cdot\Fam H)=2$.  
In particular, $H$ is dominating and unsplit. 
Hence $R_2$ is of fiber type. 
Let $\varphi_2\colon X\rightarrow Y_2$ be the contraction morphism associated to 
$R_2$. Since the restriction $\varphi_2|_{E_1}\colon E_1\rightarrow Y_2$ is a 
finite morphism, $\dim Y_2=n-1$. We note that all curves in $E_1$ are 
numerically proportional. Thus $\rho_{Y_2}=1$. This implies that $\rho_X=2$.
\end{proof}

\begin{proposition}\label{contrprop}
Let $X$ be an $n$-dimensional normal projective variety which satisfies that 
$\Pic(X)\otimes\Q=\ND(X)_\Q$.
\begin{enumerate}
\renewcommand{\theenumi}{\arabic{enumi}}
\renewcommand{\labelenumi}{$(\theenumi)$}
\item\label{contrprop1}
Assume that $\rho_X\geq 3$. Pick any extremal ray $R\subset\overline{\NE}(X)$ 
which defines the contraction morphism $\varphi\colon X\rightarrow Y$. 
Then the ray $R$ is neither of type $(n, 0)$ nor of type $(n, 1)$.
\item\label{contrprop2}
Set $m\geq 2$. Let $R_i\subset\overline{\NE}(X)$ be an extremal ray 
which defines the contraction morphism 
${\varphi}_i\colon X\rightarrow Y_i$, $C_i\subset X$ be an irreducible curve with $[C_i]\in R_i$, and 
$E_i:=\Exc(R_i)$ for any $1\leq i\leq m$. 
We assume that $E_i\cap E_j=\emptyset$ for any $1\leq i<j\leq m$.
Then we can construct the morphism $\varphi\colon X\rightarrow Y$ 
contracting all of $E_1,\ldots ,E_m$.
$($Glue ${\varphi}_1,\ldots,{\varphi}_m$ together. We note that 
$Y$ is a normal proper variety but not necessary projective.$)$
Then there is an exact sequence 
\[
\begin{CD}
0 @>>>  \sum_{i=1}^m\Q[C_i] @>>> \NC(X)_\Q 
@>{\varphi}_*>> \NC(Y)_\Q @>>> 0.\\
\end{CD}
\]
Furthermore, if $X$ is $\Q$-factorial and $R_i$ is divisorial for any $1\leq i\leq m$, 
then $Y$ is also $\Q$-factorial and hence $\rho_Y\geq 1$.
\end{enumerate}
\end{proposition}

\begin{proof}
\eqref{contrprop1} is obvious. We prove \eqref{contrprop2}. 
For $1\leq i\leq m$, let
${\psi}_i\colon X\rightarrow Z_i$ be the morphism 
contracting $E_1,\ldots,E_i$ obtained by gluing 
${\varphi}_1,\ldots,{\varphi}_i$ together (for construction, see 
\cite[Exercise 2.12]{har}). 
We note that $Z_i$ is a normal proper variety, $Y=Z_m$ and $\varphi={\psi}_m$. 
Set $Z_0:=X$ and $\psi_0:=id_X$ (the identity morphism). 
For $1\leq i\leq m$, let
${\pi}_i\colon Z_{i-1}\rightarrow Z_i$ 
be the morphism contracting (the image of) $E_i$ such that 
${\pi}_i\circ{\psi}_{i-1}={\psi}_i$. 
We remark that $\varphi_1=\psi_1=\pi_1$. 
Note that $\Pic(Z_i)\otimes\Q=\ND(Z_i)_\Q$ by Remark \ref{glue_rmk}. 
It is enough to show the exactness of 
\[
\begin{CD}
0 @>>> \ND(Z_i)_\Q @>{{\pi}_i}^{*}>> \ND(Z_{i-1})_\Q @>(\bullet\cdot C_i)>> \Q \\
\end{CD}
\]
for any $1\leq i\leq m$ to prove the exactness of the sequence in \eqref{contrprop2}.
We can assume that $2\leq i\leq m$ since 
the case $i=1$ follows from the definition of the contraction morphism. 
The injectivity of ${{\pi}_i}^{*}\colon\ND(Z_i)_\Q\rightarrow\ND(Z_{i-1})_\Q$ 
is obvious. 
Let ${\tau}_i\colon Y_i\rightarrow Z_i$ be the morphism contracting $E_1,\ldots,E_{i-1}$ 
which satisfies that the diagram commutes:
\[
\begin{CD}
X @>{{\psi}_{i-1}}>> Z_{i-1}\\
@V{{\varphi}_i}VV    @VV{{\pi}_i}V\\
Y_i @>>{{\tau}_i}>   Z_i.\\
\end{CD}
\]
Let $V_i:= Z_i\setminus({\tau}_i\circ{\varphi}_i(E_1\sqcup\ldots\sqcup E_{i-1}))$ and 
$U_i:=Z_i\setminus({\tau}_i\circ{\varphi}_i(E_i))$.
Pick any invertible sheaf $M\in\Pic(Z_{i-1})$ satisfying $(M\cdot C_i)=0$. Then 
$0=(M\cdot C_i)=({{\psi}_{i-1}}^{*}M\cdot C_i)$. 
There exists an invertible sheaf $L_1\in\Pic(Y_i)$ and a positive integer 
$t$ such that 
${{\varphi}_i}^{*}L_1\simeq{{\psi}_{i-1}}^{*}M^{\otimes t}$ by the property 
of the ray $R_i$ and the fact $\Pic(X)\otimes\Q=\ND(X)_\Q$.
Thus 
\[
M^{\otimes t} \simeq {{\psi}_{i-1}}_{*}{{\psi}_{i-1}}^{*}M^{\otimes t}
\simeq{{\psi}_{i-1}}_{*}{{\varphi}_i}^{*}L_1\simeq{{\pi}_i}^{*}{{\tau}_i}_{*}L_1.
\]
Indeed, ${\varphi}_i$ and ${\pi}_i$ are isomorphisms over $U_i$, 
and ${\psi}_{i-1}$ and ${\tau}_i$ are isomorphisms over $V_i$, respectively. 
We note that ${{\tau}_i}_{*}L_1$ is an invertible sheaf since 
${{\tau}_i}_{*}L_1{\mid}_{U_i}\simeq M^{\otimes t}{\mid}_{{{\pi}_i}^{-1}(U_i)}$ and 
${{\tau}_i}_{*}L_1{\mid}_{V_i}\simeq L_1{\mid}_{{{\tau}_i}^{-1}(V_i)}$.
Therefore we have $M^{\otimes t}\in{{\pi}_i}^{*}(\Pic(Z_i))$.
For the remaining part, see \cite[Corollary 3.18]{KoMo} for example.
\end{proof}

\begin{remark}\label{glue_rmk}
For a surjective morphism $\varphi\colon X\rightarrow Y$ between normal proper 
varieties with connected fibers, if $\Pic(X)\otimes\Q=\ND(X)_\Q$ then 
$\Pic(Y)\otimes\Q=\ND(Y)_\Q$. Indeed, for a 
numerically trivial invertible sheaf $L\in\Pic(Y)$, 
since $\varphi^*L$ is numerically trivial, there exists a positive integer $t$ such that 
$\varphi^*L^{\otimes t}\simeq\sO_X$. Thus $L^{\otimes t}\simeq\sO_Y$. 
\end{remark}

\begin{corollary}\label{glue_cor}
Let $X$ be an $n$-dimensional normal $\Q$-factorial projective variety 
such that $\Pic(X)\otimes\Q=\ND(X)_\Q$.
Assume that there exist distinct divisorial extremal rays 
$R_1,\dots,R_m\subset\overline{\NE}(X)$ which define the contraction morphisms 
$\varphi_i\colon X\rightarrow Y_i$ for all $1\leq i\leq m$ and 
$\Exc(R_i)\cap\Exc(R_j)=\emptyset$ for any $1\leq i<j\leq m$.
\begin{enumerate}
\renewcommand{\theenumi}{\arabic{enumi}}
\renewcommand{\labelenumi}{\rm{(\theenumi)}}
\item\label{glue_cor1}
If $m\geq 3$, then $\rho_X\geq 4$. 
\item\label{glue_cor2}
If $X$ is smooth and $R_i$ is of type $(n-1, b_i)^{\sm}$ 
$($for some $b_i\in\Z_{\geq 0})$
for any $1\leq i\leq m$, then $\rho_X\geq m+1$.
\end{enumerate}
\end{corollary}

\begin{proof}
Let $\varphi\colon X\rightarrow Y$ be the morphism which is the gluing morphism 
of $\varphi_1,\dots,\varphi_m$ contracting $\Exc(R_1),\dots,\Exc(R_m)$ 
as in Proposition \ref{contrprop} \eqref{contrprop2}. 
Let $C_i\subset X$ be an irreducible and reduced curve with $[C_i]\in R_i$ for 
$1\leq i\leq m$. 

\eqref{glue_cor1}
We can assume that the classes $[C_1]$, $[C_2]$, $[C_3]$ are linearly independent 
in $\NC(X)$. By Proposition \ref{contrprop} \eqref{contrprop2}, $Y$ is $\Q$-factorial 
and $1\leq\rho_Y\leq\rho_X-3$. 

\eqref{glue_cor2}
In this case, $Y$ is a smooth proper variety and $\rho_X=m+\rho_Y\geq m+1$. 
\end{proof}

We recall Wi\'sniewski's theorem on the bounds of the length of extremal rays. 

\begin{thm}[{\cite[Theorem 1.1]{wisn}}]\label{wisinieq}
Let $X$ be a smooth projective variety, 
$R\in \overline{\NE}(X)$ be a $K_X$-negative extremal ray 
and $\cont_R\colon X\rightarrow Y$ be the associated contraction morphism. 
Then for every irreducible component $E\subset\Exc(R)$, we have 
\[
l(R)\leq \dim X+1-2{\codim}_XE-\dim({\cont}_R(E)).
\]
\end{thm}


\subsection{Characterizations of the products of projective spaces}


We give several criteria so that a given smooth projective variety is isomorphic to 
the products of projective spaces. 

\begin{thm}[{\cite[Theorem 2.16]{kebekus}}]\label{kebthm}
Let $X$ be a normal projective variety and $H$ be a dominating and locally unsplit 
family of rational curves on $X$. For general $x\in X$, consider the rational map
\[
\tau_x\colon\tilde{H_x}\dasharrow\pr(T_X{\arrowvert}_x^\vee)
\]
defined by 
\[
[l]\mapsto{\pr(T_l{\arrowvert}_x^\vee)}.
\]
Then the rational map ${\tau}_x$ is a finite morphism.
\end{thm}

\begin{definition}[Variety of Minimal Rational Tangents]\label{VMRT}
Under the assumption in Theorem \ref{kebthm}, the finite morphism 
${\tau}_x$ is called the \emph{tangent morphism}; its image 
${\sC}_x:={\tau}_x(\tilde{H_x})\subset\pr(T_X{\arrowvert}_x^\vee)$ is called 
the \emph{variety of minimal rational tangents}, or shortly \emph{VMRT}, of $H$ at $x$.
\end{definition}

Araujo \cite{araujo} showed a criterion for varieties being isomorphic to the products of 
projective spaces in terms of VMRT.

\begin{thm}[{\cite[Theorem 1.3]{araujo}}]\label{araujopr}
Let $X$ be an $n$-dimensional smooth projective variety with $k$ distinct 
dominating and unsplit family of rational curves $H_1,\ldots,H_k$ on $X$.
Suppose that, for a general $x\in X$, the associated VMRT of $H_i$ at $x$ are 
linear subspaces of dimension $d_i-1$ in $\pr(T_X{\arrowvert}_x^\vee)$ 
such that $\sum_{i=1}^{k}d_i=n$. Then $X\simeq\prod_{i=1}^{k}{\pr}^{d_i}$.
\end{thm}

We give another criterion for varieties being isomorphic to the products of 
projective spaces in terms of length of extremal rays.

\begin{thm}\label{prlngth}
Let $X$ be an $n$-dimensional smooth projective variety with 
$n=\sum_{i=1}^{k}d_i$, where $d_1,\ldots,d_k\in\Z_{>0}$.
Assume that there exist distinct $K_X$-negative extremal rays $R_1,\ldots,R_k\subset\overline{\NE}(X)$ 
such that $R_i$ are of fiber type and $l(R_i)\geq d_i+1$ for all $1\leq i\leq k$.
Then $X\simeq\prod_{i=1}^k\pr^{d_i}$.
\end{thm}

\begin{proof}
Let $\varphi_i\colon X\rightarrow Y_i$ be the contraction morphism associated to 
$R_i$ and $e_i:=\dim X-\dim Y_i$ for $1\leq i\leq k$. 
We have $\sum_{i=1}^ke_i\leq n$ 
and $e_i\geq\l(R_i)-1$ for any $i$ by \cite[Theorem 2.2]{wisn} and 
Theorem \ref{wisinieq}. 
Hence we obtain the inequality
\[
n\geq\sum_{i=1}^ke_i\geq\sum_{i=1}^k(l(R_i)-1)\geq\sum_{i=1}^kd_i=n.
\]
Therefore $e_i=l(R_i)-1=d_i$ for any $i$.
Let $F_i$ be a general fiber of ${\varphi}_i$. 
Then $F_i$ is a $d_i$-dimensional Fano manifold such that any rational curve 
$l_i$ in $F_i$ satisfies that $(-K_{F_i}\cdot l_i)\geq d_i+1$. 
Hence $F_i\simeq\pr^{d_i}$ by \cite{CMSB}.
Let $H_i$ be the family of rational curves on $X$ 
containing points parameterizing \emph{lines} in $F_i\simeq\pr^{d_i}$.
Then $H_i$ is a dominating and unsplit family
since $(-K_X\cdot\Fam H_i)=d_i+1=l(R_i)$.
We consider ${\sC}^i_x\subset\pr(T_X{\arrowvert}_x^\vee)$ for $x\in F_i$, 
which is a VMRT of $H_i$ at $x$. 
We have 
${\sC}^i_x=\pr(T_{F_i}{\arrowvert}_x^\vee)\subset\pr(T_X{\arrowvert}_x^\vee)$; 
a linear subspace of dimension $d_i-1$.
By Theorem \ref{araujopr}, $X\simeq\prod_{i=1}^k\pr^{d_i}$.
\end{proof}

We also give a criterion for varieties being isomorphic to the 
product of two projective spaces in terms of extremal rays. 

\begin{proposition}\label{tworay}
Let $X$ be an $n$-dimensional smooth projective variety.
If there exist distinct $K_X$-negative extremal rays $R_1,R_2\subset\overline{\NE}(X)$ 
such that the intersection $\Exc(R_1)\cap\Exc(R_2)$ is not empty. 
Then we have 
\[
(l(R_1)-1)+(l(R_2)-1)\leq n,
\]
and the equality holds if and only if $X\simeq\pr^{l(R_1)-1}\times\pr^{l(R_2)-1}$.
\end{proposition}

\begin{proof}
We fix an arbitrary point $x\in\Exc (R_1)\cap\Exc(R_2)$. 
For $i=1$, $2$, let 
$\varphi_i\colon X\rightarrow Y_i$ be the contraction morphism associated to $R_i$ 
and set $y_i:={\varphi}_i(x)\in Y_i$. 
Let 
$C_i\subset X$ be a rational curve which satisfies that 
\begin{enumerate}
\renewcommand{\theenumi}{\arabic{enumi}}
\renewcommand{\labelenumi}{(\theenumi)}
\item\label{tworay1}
$x\in C_i\text{ and }[C_i]\in R_i$,
\item\label{tworay2}
$(-K_X\cdot C_i)$ is minimal among satisfying \eqref{tworay1}.
\end{enumerate}
Let $H_i$ be a family of rational curves on $X$ containing $[C_i]\in\Rat(X)$. 
Then ${(H_i)}_x$ is projective by construction. 
Hence we have 
\begin{eqnarray*}
\dim{\varphi}_i^{-1}(y_i) & \geq & \dim\Lo((H_i)_x)\\
 & \geq & (n-\dim\Lo(H_i))+(-K_X\cdot\Fam H_i)-1\\
 & \geq & (-K_X\cdot\Fam H_i)-1\geq l(R_i)-1
\end{eqnarray*}
by Proposition \ref{famineq}. 
We note that the intersection ${\varphi}_1^{-1}(y_1)\cap{\varphi}_2^{-1}(y_2)$ 
does not contain curves since the rays $R_1$ and $R_2$ are distinct. 
Hence $\dim({\varphi}_1^{-1}(y_1)\cap{\varphi}_2^{-1}(y_2))=0$. 
Thus $n\geq \dim{\varphi}_1^{-1}(y_1)+\dim{\varphi}_2^{-1}(y_2)$.
Hence $n\geq (l(R_1)-1)+(l(R_2)-1)$.
If $n=(l(R_1)-1)+(l(R_2)-1)$, then $H_i$ is dominating and unsplit for each $i=1$, $2$ 
since $(-K_X\cdot\Fam H_i)=l(R_i)$ and $\dim\Lo(H_i)=n$. 
Therefore one has $X\simeq\pr^{l(R_1)-1}\times\pr^{l(R_2)-1}$ 
by \cite[Theorem 1.1]{Occ}.
\end{proof}

\begin{corollary}\label{rhotwo}
Let $X$ be an $n$-dimensional Fano manifold with $\rho_X=2$.
Then $\NE(X)$ is spanned by two extremal rays, say $R_1$ and $R_2$.
If, at least, one of $R_1$ and $R_2$ is not small, then we have 
\[
(l(R_1)-1)+(l(R_2)-1)\leq n,
\]
and the equality holds if and only if 
$X\simeq\pr^{l(R_1)-1}\times\pr^{l(R_2)-1}$.
\end{corollary}

\begin{proof}
For $i=1$, $2$, let $\varphi_i\colon X\rightarrow Y_i$ be the contraction morphism 
associated to $R_i$ and $E_i:=\Exc(R_i)$. 
It is enough to show that $E_1\cap E_2\ne\emptyset$ by Proposition \ref{tworay}.
We can assume that $R_1$ is divisorial.
Then we have $(E_1\cdot R_1)<0$. 
Thus $(E_1\cdot R_2)>0$ holds since $E_1$ is a prime divisor and 
since 
$R_1$ and $R_2$ span the cone $\NE(X)$. 
Hence 
$E_1\cap E_2\ne \emptyset$.
\end{proof}

\section{Fano manifolds having special extremal rays}

In this section, we see several classification results of Fano manifolds having special 
extremal rays and calculate $s(X)$ for such Fano manifolds $X$. 

\begin{thm}[{\cite[Proposition 3.1 and Theorem 1.1]{divcurve}}]\label{casa}
Let $X$ be an $n$-dimensional Fano manifold 
and $R\subset\NE(X)$ be an extremal ray.
\begin{enumerate}
\renewcommand{\theenumi}{\arabic{enumi}}
\renewcommand{\labelenumi}{$(\theenumi)$}
\item\label{casa1}
If $n\geq 3$ and $R$ is of type $(n-1,0)$,
then ${\rho}_X\leq 3$.
\item\label{casa2}
If $n\geq 4$ and $R$ is of type $(n-1,1)$,
then ${\rho}_X\leq 5$.
\end{enumerate}
\end{thm}

\begin{thm}[{\cite[Theorem 5.1]{AO}}]\label{AnOc}
Let $X$ be an $n$-dimensional smooth projective variety and 
$R\subset \overline{\NE}(X)$ be a $K_X$-negative extremal ray 
of type $(n-1, m)$ which satisfies 
that $l(R)=n-1-m$ and all nontrivial fibers of the associated contraction morphism 
of $R$ are of equi-dimensional.
Then $R$ is of type ${(n-1,m)}^{\sm}$.
\end{thm}

\begin{proposition}[{\cite[Proposition 5]{tsu2} 
(and \cite[Theorem 5.1]{AO})}]\label{tsusmooth}
Let $X$ be an $n$-dimensional Fano manifold with $n\geq 4$.
Assume that there exist distinct extremal rays $R_1,R_2\subset\NE(X)$ such that 
$R_i$ is of type $(n-1, 1)$ and $l(R_i)=n-2$ for each $i=1$, $2$.
Then $\Exc(R_1)\cap\Exc(R_2)=\emptyset$.
\end{proposition}

\begin{thm}[{\cite[Theorem 1.1]{divpt}}]\label{BCW}
Let $Y$ be an $n$-dimensional smooth projective variety with $n\geq 3$ and 
$a\in Y$ be a $($closed$)$ point.
Then $X:={\Bl}_a(Y)$ is a Fano manifold if and only if one of the following holds:
\begin{enumerate}
\renewcommand{\theenumi}{\roman{enumi}}
\renewcommand{\labelenumi}{\rm{(\theenumi)}}
\item\label{BCW1}
$Y\simeq\pr^n$ and $a\in Y$ is an arbitrary point. 
\item\label{BCW2}
$Y\simeq\Q^n$ and $a\in Y$ is an arbitrary point.
\item\label{BCW3}
$Y\simeq V_d$ with $1\leq d\leq n$ and $a\notin H'$ $($the strict transform of $H$$)$ 
with $V_d:={\Bl}_Z(\pr^n)$, where $H\subset\pr^n$ is a hyperplane and 
$Z\subset H$ is a smooth subvariety of dimension $n-2$ and degree $d$.
\end{enumerate}
\end{thm}

\begin{remark}\label{calBCW}
We have the following properties by easy calculations. 
\begin{enumerate}
\renewcommand{\theenumi}{\roman{enumi}}
\renewcommand{\labelenumi}{\rm{(\theenumi)}}
\item\label{calBCW1}
If $X=\Bl_a(Y)$ is in Theorem \ref{BCW} \eqref{BCW1}, then 
\begin{eqnarray*}
\NE(X) & = & {\R}_{\geq 0}[f]+{\R}_{\geq 0}[g],\\
(-K_X\cdot f) & = & 2, \\
(-K_X\cdot g) & = & n-1
\end{eqnarray*}
hold, where $f$ is the strict transform of a line on $Y=\pr^n$ passing through $a$ 
and $g$ is a line in the exceptional divisor ($\simeq\pr^{n-1}$) of $X\rightarrow Y$. 
Thus $s(X)=n-1$. 
\item\label{calBCW2}
If $X=\Bl_a(Y)$ is in Theorem \ref{BCW} \eqref{BCW2}, then 
\begin{eqnarray*}
\NE(X) & = & {\R}_{\geq 0}[f]+{\R}_{\geq 0}[g],\\
(-K_X\cdot f) & = & 1, \\
(-K_X\cdot g) & = & n-1
\end{eqnarray*}
hold, where $f$ is the strict transform of a line on $Y=\Q^n$ passing through $a$ 
and $g$ is a line in the exceptional divisor ($\simeq\pr^{n-1}$) of $X\rightarrow Y$. 
Thus $s(X)=n-2$. 
\item\label{calBCW3}
If $X=\Bl_a(Y)$ is in Theorem \ref{BCW} \eqref{BCW3}, then 
\begin{eqnarray*}
\NE(X) & = & {\R}_{\geq 0}[f]+{\R}_{\geq 0}[g]+{\R}_{\geq 0}[l]+{\R}_{\geq 0}[m],\\ 
l & \equiv & m+g+(1-d)f\text{ in }\NC(X),\\
(-K_X\cdot f) & = & 1,\quad (-K_X\cdot g)=1,\\
(-K_X\cdot l) & = & n+1-d,\quad (-K_X\cdot m)=1 
\end{eqnarray*}
hold, where 
$f \subset X$ is a fiber over $Z$,
$g \subset X$ is a line in a fiber over $a$,
$l \subset X$ is a line in $H'$, and 
$m \subset X$ is a strict transform of a line passing through $a$ and a point in $Z$.
Thus if $d=1$ then $s(X)=n-2$, but if $d>1$ then $s(X)=2n-2-d$. 
We note that if $d=2$, then $X$ is isomorphic to $\Bl_{p,q}({\Q}^{n})$ with 
$\overline{pq}\not\subset\Q^n(\subset\pr^{n+1})$ 
(see \cite[Corollaire 1.2]{divpt}) and $s(X)=2n-4$. 
\end{enumerate}
\end{remark}

\begin{thm}[{\cite{quasi, tsu1, tsu3}}]\label{tsucurve}
Let $Y$ be an $n$-dimensional smooth projective variety with $n\geq 4$, 
$C\subset Y$ be a smooth curve, $X:={\Bl}_C(Y)$, and $E$ be the exceptional 
divisor of the morphism $X\rightarrow Y$.
We assume that $X$ is a Fano manifold.
\begin{enumerate}
\renewcommand{\theenumi}{\arabic{enumi}}
\renewcommand{\labelenumi}{$(\theenumi)$}
\item\label{tsucurve1}
If ${\rho}_X=5$, then one of the following holds: 
\begin{enumerate}
\renewcommand{\theenumii}{\roman{enumii}}
\renewcommand{\labelenumii}{$\rm{(\theenumii)}$}
\item\label{tsucurve11}
$Y\simeq{\Bl}_{\{p\}\cup\{q\}\cup{\pr}^{n-2}}(\pr^n)$ with $\pr^{n-2}\cap\overline{pq}=\emptyset$ and 
$C$ is the strict transform of $\overline{pq}$. 
\item\label{tsucurve12}
$Y\simeq{\Bl}_{\{p\}\cup\{q\}\cup\Q^{n-2}}(\pr^n)$ with $\Q^{n-2}\cap\overline{pq}=\emptyset$ and 
$C$ is the strict transform of $\overline{pq}$.
\end{enumerate}

\item\label{tsucurve2}
Assume that there exists an extremal ray $R\subset\NE(X)$ of fiber type with 
$l(R)\geq 2$ and $(E\cdot R)>0$.
\begin{itemize}
\item\label{tsucurve21}
If $R$ is of type $(n, n-2)$, then ${\rho}_X=2$.
\item\label{tsucurve22}
If $R$ is of type $(n, n-1)$, then the pair of $(Y, C)$ is one of the following:
\begin{enumerate}
\renewcommand{\theenumii}{\roman{enumii}}
\renewcommand{\labelenumii}{$\rm{(\theenumii)}$}
\item\label{tsucurve221}
$Y\simeq\Q^n$ and $C$ is a line in $\Q^n\subset\pr^{n+1}$.
\item\label{tsucurve222}
$Y\simeq\pr^1\times\pr^{n-1}$ and $C$ is a fiber of the second projection.
\item\label{tsucurve223}
$Y\simeq{\Bl}_{\pr^{n-2}}(\pr^n)$ and $C$ is the strict transform of a line in $\pr^n$ 
disjoint from $\pr^{n-2}$.
\item\label{tsucurve224}
$Y\simeq{\Bl}_{\pr^{n-2}}(\pr^n)$ and $C$ is a fiber of the blowing up.  \item\label{tsucurve225}
$Y\simeq{\pr}_{\pr^1}({\sO}_{\pr^1}\oplus{{\sO}_{\pr^1}(1)}^{\oplus n-1})$ 
and $C$ is the section of $\pr^{n-1}$-bundle over $\pr^1$ whose normal bundle 
${\mathcal{N}}_{C/Y}$ is isomorphic to ${{\sO}_{\pr^1}(-1)}^{\oplus n-1}$.
\end{enumerate}
\end{itemize}
\item\label{tsucurve3}
Assume that there exists an extremal ray $R\subset\NE(X)$ of fiber type with 
$(E\cdot R)=0$. Let $\varphi\colon X\rightarrow Z$ be the contraction morphism 
associated to $R$. Then $R$ is of type $(n, n-1)$, $C\simeq\pr^1$, 
$E\simeq\pr^1\times\pr^{n-2}$, $E=\varphi^*D$ and 
$Z$ is factorial, where $D:=\varphi(E)$ with the 
reduced structure. Furthermore, if $n=4$, then there exists an extremal ray 
$R_Z\subset\NE(Z)$ with the associated contraction morphism 
$\varphi_Z\colon Z\rightarrow W$ such that $\varphi_Z$ maps $D$ to a point.
\end{enumerate}
\end{thm}

\begin{proof}
\eqref{tsucurve1} and \eqref{tsucurve2} follow from 
\cite[Theorem 1]{tsu1} and \cite[Propositions 3, 4]{tsu3}. 
We prove \eqref{tsucurve3}. 
The ray $R$ is of type $(n, n-1)$, $E=\varphi^*D$ and $Z$ is factorial 
by the fact $\dim D\geq n-2$ 
and by \cite[Lemmas 3.9 (i), 3.10 (i)]{quasi}. 
Moreover, $C\simeq\pr^1$ since a one-dimensional fiber of 
$\varphi$ in $E$ maps $X\rightarrow Y$ onto $C$. 
We know that $E\simeq\pr^1\times\pr^{n-2}$ since $E\simeq\pr_C(\sN^\vee_{C/Y})$ 
and $\dim E>\dim D$, where $\sN_{C/Y}$ is the normal bundle of $C$ in $Y$. 
The cone $\NE(Z)$ is closed since $\NE(X)$ is so. 
Assume that $n=4$. Then the existence of the ray 
$R_Z\subset\NE(Z)$ follows from \cite[Theorem 4.1 (ii)]{quasi}. 
\end{proof}

\begin{remark}\label{caltsu}
We have the following properties by easy calculations. 
\begin{enumerate}
\renewcommand{\theenumi}{\arabic{enumi}}
\renewcommand{\labelenumi}{$(\theenumi)$}
\item\label{caltsu1}
\begin{enumerate}
\renewcommand{\theenumii}{\roman{enumii}}
\renewcommand{\labelenumii}{$\rm{(\theenumii)}$}
\item\label{caltsu11}
If $X=\Bl_C(Y)$ is in Theorem \ref{tsucurve} \eqref{tsucurve1} $(\rm{i})$, then 
\begin{eqnarray*}
\NE(X) & = & {\R}_{\geq 0}[e]+{\R}_{\geq 0}[f]+{\R}_{\geq 0}[g]+{\R}_{\geq 0}[h]\\
 & + & {\R}_{\geq 0}[k]+{\R}_{\geq 0}[l]+{\R}_{\geq 0}[m],\\
(-K_X\cdot e) & = & n-2,\quad (-K_X\cdot f)=1,\quad(-K_X\cdot g)=1, \\
(-K_X\cdot h) & = & 1,\,\, (-K_X\cdot k)=1,\,\, (-K_X\cdot l)=1,\,\, (-K_X\cdot m)=1,  
\end{eqnarray*}
and $\NE(X)$ is exactly spanned by the above seven rays, where 
\begin{itemize}
\item
$e$ is a nontrivial fiber of the morphism $X\rightarrow Y$, 
\item
$f$ is the strict transform of a line 
in the exceptional divisor over $p$, 
\item
$g$ is the strict transform of a line 
in the exceptional divisor over $q$, 
\item
$h$ is a fiber over $\pr^{n-2}$, 
\item
$k$ is a fiber of $E\simeq C\times\pr^{n-2}\rightarrow\pr^{n-2}$, 
where $E$ is the exceptional divisor of 
$X\rightarrow Y$, 
\item
$l$ is the strict transform of a line in $\pr^n$ passing through $p$ and $\pr^{n-2}$, 
\item
$m$ is the strict transform of a line in $\pr^n$ passing through $q$ and $\pr^{n-2}$. 
\end{itemize}
Thus $s(X)=n-3$. 
\item\label{caltsu12}
If $X=\Bl_C(Y)$ is in Theorem \ref{tsucurve} \eqref{tsucurve1} $(\rm{ii})$, then 
\begin{eqnarray*}
\NE(X) & = & {\R}_{\geq 0}[e]+{\R}_{\geq 0}[f]+{\R}_{\geq 0}[g]+{\R}_{\geq 0}[h]\\
 & + & {\R}_{\geq 0}[j]+{\R}_{\geq 0}[k]+{\R}_{\geq 0}[l]+{\R}_{\geq 0}[m],\\
(-K_X\cdot e) & = & n-2,\, (-K_X\cdot f)=1,\, (-K_X\cdot g)=1, (-K_X\cdot h)=1,\\
(-K_X\cdot j) & = & 1,\,\, (-K_X\cdot k)=1,\,\, (-K_X\cdot l)=1,\,\, (-K_X\cdot m)=1,  
\end{eqnarray*}
and $\NE(X)$ is exactly spanned by the above eight rays, where 
\begin{itemize}
\item
$e$ is a nontrivial fiber of the morphism $X\rightarrow Y$, 
\item
$f$ is the strict transform of a line 
in the exceptional divisor over $p$, 
\item
$g$ is the strict transform of a line 
in the exceptional divisor over $q$, 
\item
$h$ is a fiber over $\Q^{n-2}$, 
\item
$j$ is the strict transform of a line in $\pr^n$ intersects $\overline{pq}$ with each other 
and is contained in a unique hyperplane in $\pr^n$ which contains $\Q^{n-2}$, 
\item
$k$ is a fiber of $E\simeq C\times\pr^{n-2}\rightarrow\pr^{n-2}$, 
where $E$ is the exceptional divisor of $X\rightarrow Y$, 
\item
$l$ is the strict transform of a line in $\pr^n$ passing through $p$ and $\Q^{n-2}$, 
\item
$m$ is the strict transform of a line in $\pr^n$ passing through $q$ and $\Q^{n-2}$. 
\end{itemize}
Thus $s(X)=n-3$. 

\end{enumerate}
\item\label{caltsu2}
\begin{enumerate}
\renewcommand{\theenumii}{\roman{enumii}}
\renewcommand{\labelenumii}{$\rm{(\theenumii)}$}
\item\label{caltsu221}
If $X=\Bl_C(Y)$ is in Theorem \ref{tsucurve} \eqref{tsucurve2} $(\rm{i})$, then 
$\rho_X=2$. Thus $s(X)<n$ by Corollary \ref{rhotwo}. 
\item\label{caltsu222}
If $X=\Bl_C(Y)$ is in Theorem \ref{tsucurve} \eqref{tsucurve2} $(\rm{ii})$, then 
\begin{eqnarray*}
\NE(X) & = & {\R}_{\geq 0}[f]+{\R}_{\geq 0}[g]+{\R}_{\geq 0}[h],\\
(-K_X\cdot f) & = & n-2,\quad (-K_X\cdot g)=2,\quad(-K_X\cdot h)=2
\end{eqnarray*}
hold, where $f$ is a nontrivial fiber of $X\rightarrow Y$, $g$ is the strict transform 
of a general fiber of the first projection $Y=\pr^1\times\pr^{n-1}\rightarrow\pr^{n-1}$ 
and $h$ is the strict transform of a line in the second projection 
$Y=\pr^1\times\pr^{n-1}\rightarrow\pr^1$ passing through $C$. Thus $s(X)=n-1$. 
\item\label{caltsu223}
If $X=\Bl_C(Y)$ is in Theorem \ref{tsucurve} \eqref{tsucurve2} $(\rm{iii})$, then 
\begin{eqnarray*}
\NE(X) & = & {\R}_{\geq 0}[f]+{\R}_{\geq 0}[g]+{\R}_{\geq 0}[h],\\
(-K_X\cdot f) & = & n-2,\quad (-K_X\cdot g)=1,\quad(-K_X\cdot h)=2
\end{eqnarray*}
hold, where $f$ is a nontrivial fiber of $X\rightarrow Y$, $g$ is a fiber over $\pr^{n-2}$
and $h$ is the strict transform of a line in $\pr^n$ passing through $C$ and $\pr^{n-2}$. 
Thus $s(X)=n-2$. 
\item\label{caltsu224}
If $X=\Bl_C(Y)$ is in Theorem \ref{tsucurve} \eqref{tsucurve2} $(\rm{iv})$, then 
\begin{eqnarray*}
\NE(X) & = & {\R}_{\geq 0}[f]+{\R}_{\geq 0}[g]+{\R}_{\geq 0}[h],\\
(-K_X\cdot f) & = & n-2,\quad (-K_X\cdot g)=1,\quad(-K_X\cdot h)=2
\end{eqnarray*}
hold, where $f$ is a nontrivial fiber of $X\rightarrow Y$, $g$ is a general fiber over $\pr^{n-2}$
and $h$ is the strict transform of a line in $\pr^n$ passing through $\pr^{n-2}$ 
and the image of $C$ in $\pr^n$. Thus $s(X)=n-2$. 
\item\label{caltsu225}
If $X=\Bl_C(Y)$ is in Theorem \ref{tsucurve} \eqref{tsucurve2} $(\rm{v})$, then 
\begin{eqnarray*}
\NE(X) & = & {\R}_{\geq 0}[f]+{\R}_{\geq 0}[g]+{\R}_{\geq 0}[h],\\
(-K_X\cdot f) & = & n-2,\quad (-K_X\cdot g)=1,\quad(-K_X\cdot h)=2
\end{eqnarray*}
hold, where $f$ is a nontrivial fiber of $X\rightarrow Y$, $g$ is a fiber of 
$E\simeq C\times\pr^{n-2}\rightarrow\pr^{n-2}$, 
where $E$ is the exceptional divisor of $X\rightarrow Y$, 
and $h$ is the strict transform of a line in a fiber of $Y\rightarrow\pr^1$ 
passing through $C$. Thus $s(X)=n-2$. 
\end{enumerate}
\end{enumerate}
\end{remark}


\section{Proof of Theorem \ref{mainthm}}

In this section, we prove Theorem \ref{mainthm}. 
If an $n$-dimensional Fano manifold $X$ satisfies that 
$s(X)\geq n$ and ${\rho}_X=1$, then $s(X)=n$ and $X\simeq\pr^n$ by \cite{CMSB}. 
Hence we can consider only the Fano manifolds $X$ with $\rho_X\geq 2$.

\subsection{Proof of Theorem \ref{mainthm} \eqref{mainthm1}}

We can assume that $n=3$ since the case $n\leq 2$ is trivial. 
We prove the assertion without using 
the result \cite{MoMu} of complete classification of $3$-dimensional 
Fano manifolds $X$ with ${\rho}_X\geq 2$.
Let $X$ be a $3$-dimensional Fano manifold with $s(X)\geq 3$.
We can assume that ${\rho}_X\geq 3$ by Corollary \ref{rhotwo}.
By Theorem \ref{wisinieq}, Proposition \ref{contrprop} \eqref{contrprop1} and Theorem \ref{AnOc},
any extremal ray $R\subset\NE(X)$ with $l(R)\geq 2$ 
satisfies one of the following:
\begin{enumerate}
\renewcommand{\theenumi}{\Alph{enumi}}
\renewcommand{\labelenumi}{\rm{(\theenumi)}}
\item\label{1A}
$R$ is of type ${(2, 0)}^{\sm}$ and $l(R)=2$.
\item\label{1B}
$R$ is of type $(3, 2)$ and $l(R)=2$.
\end{enumerate}
(We note that this result directly follows from 
\cite[Theorems 3.3, 3.5]{Mor}.)
If there exists an extremal ray $R\subset\NE(X)$ of type \eqref{1A}, then 
$X\simeq{\Bl}_a(V_d)$ with $1\leq d\leq 3$ by Theorem \ref{BCW}, thus $s(X)<3$ by Remark \ref{calBCW} \eqref{calBCW3}.
If there exist distinct extremal rays $R_1$, $R_2$ and $R_3\subset\NE(X)$ such that 
all of them are of type \eqref{1B}, then 
$X\simeq\pr^1\times\pr^1\times\pr^1$ by Theorem \ref{prlngth}. 
Therefore we have completed the proof of Theorem \ref{mainthm} \eqref{mainthm1}. 

\subsection{Proof of Theorem \ref{mainthm} \eqref{mainthm2}}

Let $X$ be a $4$-dimensional Fano manifold with $s(X)\geq 4$. 
We can assume that $\rho_X\geq 3$ by Corollary \ref{rhotwo}. 
(We note that if ${\rho}_X=2$ and both extremal rays are small,
then $s(X)=0$.)
By Theorem \ref{wisinieq}, Proposition \ref{contrprop} \eqref{contrprop1} 
and Theorem \ref{AnOc}, any extremal ray $R\subset\NE(X)$ with $l(R)\geq 2$ 
satisfies one of the following: 

\begin{enumerate}
\renewcommand{\theenumi}{\Alph{enumi}}
\renewcommand{\labelenumi}{\rm{(\theenumi)}}
\item\label{2A}
$R$ is of type ${(3, 0)}^{\sm}$ and $l(R)=3$.
\item\label{2B}
$R$ is of type $(3, 0)$ and $l(R)=2$.
\item\label{2C}
$R$ is of type ${(3, 1)}^{\sm}$ and $l(R)=2$.
\item\label{2D}
$R$ is of type $(4, 3)$ and $l(R)=2$.
\item\label{2E}
$R$ is of type $(4, 2)$ and $l(R)=3$.
\item\label{2F}
$R$ is of type $(4, 2)$ and $l(R)=2$.
\end{enumerate}
We note that all two distinct divisorial extremal rays $R_1$, $R_2$ with 
$l(R_1)$, $l(R_2)\geq 2$ satisfy that $\Exc(R_1)\cap\Exc(R_2)=\emptyset$ 
by Propositions \ref{divptprop} and \ref{tsusmooth}. 

Assume that there exists an extremal ray $R$ of type \eqref{2A}. Then 
$X\simeq{\Bl}_a(V_2)\simeq{\Bl}_{p,q}(\Q^4)$ and $s(X)=4$
by Theorem \ref{BCW} and Remark \ref{calBCW} \eqref{calBCW3}. 
Assume that there exists an extremal ray $R$ of type \eqref{2B} and 
there is no extremal ray of type \eqref{2A}. Then $\rho_X=3$ and 
any other extremal ray $R'$ with 
$l(R')\geq 2$ is of type \eqref{2B} or \eqref{2C} by Proposition \ref{divptprop} 
and Theorem \ref{casa} \eqref{casa1}. Since $s(X)\geq 4$, there exist distinct extremal 
rays $R_1$, $R_2$, $R_3$ apart from $R$ such that each of them is of type 
\eqref{2B} or \eqref{2C}. This contradicts to Corollary \ref{glue_cor} \eqref{glue_cor1}. 
Hence we can assume that any extremal ray $R$ with $l(R)\geq 2$ is of type 
\eqref{2C}, \eqref{2D}, \eqref{2E} or \eqref{2F}. 

Assume that there exists an extremal ray $R_1$ of type \eqref{2C}. 
We have $\rho_X\leq 4$ by Theorems \ref{casa} \eqref{casa2},
\ref{tsucurve} \eqref{tsucurve1} 
and Remark \ref{caltsu} \eqref{caltsu1}. 
By Corollary \ref{glue_cor} \eqref{glue_cor1}, the number of extremal rays of type 
\eqref{2C} is at most three. Since $s(X)\geq 4$, 
there exists an extremal ray $R_0$ of fiber type and $l(R_0)\geq 2$. 
Then $(\Exc(R_1)\cdot R_0)=0$ 
and $R_0$ is of type \eqref{2D} 
by Theorem \ref{tsucurve} \eqref{tsucurve2}, \eqref{tsucurve3} and Remark 
\ref{caltsu} \eqref{caltsu2}. Moreover, any extremal ray $R'$ of fiber type 
apart from $R_0$ satisfies that $(\Exc(R_1)\cdot R')>0$. 
Indeed, by Theorem \ref{tsucurve} \eqref{tsucurve3}, if $(\Exc(R_1)\cdot R')=0$ then 
$R'$ contains the class of a fiber of the morphism 
$\Exc(R_1)\simeq\pr^1\times\pr^2\rightarrow\pr^2$. This implies that $R'=R_0$, 
which leads to a contradiction. 
Thus $l(R')=1$ by Theorem \ref{tsucurve} \eqref{tsucurve2} 
and Remark \ref{caltsu} \eqref{caltsu2}. 
Since $s(X)\geq 4$, there exist distinct extremal rays $R_2$, $R_3$ 
apart from $R_1$ such that 
$R_2$, $R_3$ are of type \eqref{2C}. We note that $\rho_X=4$ by 
Corollary \ref{glue_cor}. 
Let $\varphi\colon X\rightarrow Y$ 
be the contraction morphism associated to $R_0$ and 
set $D_i:=\varphi(\Exc(R_i))$ for $1\leq i\leq 3$. Since $\Exc(R_i)=\varphi^*D_i$, 
$D_i\cap D_j=\emptyset$ for $1\leq i<j\leq 3$. 
By Theorem \ref{tsucurve} \eqref{tsucurve3}, for any $1\leq i\leq 3$, 
there exists a contraction morphism 
$\psi_i\colon Y\rightarrow Z_i$ associated to an extremal ray 
$R_Z^i\subset\NE(Y)$ such that $\psi_i(D_i)$ is a point. 
Since $\rho_Y=3$, each ray $R_Z^i$ is divisorial by 
Proposition \ref{contrprop} \eqref{contrprop1}. 
However, this contradicts to 
Corollary \ref{glue_cor} \eqref{glue_cor1}. 

Therefore, we can assume that any extremal ray $R$ with $l(R)\geq 2$ is of fiber type. 
Since $s(X)\geq 4$, there exist distinct extremal rays $R_1,\dots,R_m$ of fiber type 
such that $\sum_{i=1}^m(l(R_i)-1)\geq 4$. By Theorem \ref{prlngth}, 
$\sum_{i=1}^m(l(R_i)-1)=4$ and $X\simeq\prod_{i=1}^m\pr^{l(R_i)-1}$.

As a consequence, we have completed the proof of 
Theorem \ref{mainthm} \eqref{mainthm2}.

\smallskip

\noindent K.\ Fujita

Research Institute for Mathematical Sciences (RIMS),
Kyoto University, 

Oiwake-cho, Kitashirakawa, Sakyo-ku, Kyoto 606-8502, Japan 

fujita@kurims.kyoto-u.ac.jp


\begin{thebibliography}{99}

\bibitem[ACO04]{ACO}
M.\ Andreatta, E.\ Chierici, and G.\ Occhetta, \emph{Generalized Mukai
  conjecture for special Fano varieties}, Cent.\ Eur.\ J.\ Math.\ \textbf{2} (2004), 
no.\ 2, 272--293.

\bibitem[AO02]{AO}
M.\ Andreatta and G.\ Occhetta, \emph{Special rays in the Mori cone of a projective variety}, 
 Nagoya Math.\ J.\  \textbf{168}  (2002), 127--137.

\bibitem[Ara06]{araujo}
C.\ Araujo, \emph{Rational curves of minimal degree and characterizations of projective spaces},
 Math.\ Ann.\ \textbf{335} (2006), no.\ 4, 937--951.

\bibitem[BCDD03]{gm}
L.\ Bonavero, C.\ Casagrande, O.\ Debarre and S.\ Druel, \emph{Sur 
une conjecture de Mukai}. Comment.\ Math.\ Helv.\ \textbf{78} (2003), no.\ 3, 616--626. 

\bibitem[BCW02]{divpt}
L.\ Bonavero, F.\ Campana, and J.\ A.\ Wi{\'s}niewski, \emph{Vari{\'e}t{\'e}s
  projectives complexes dont l'{\'e}clat{\'e}e en un point est de Fano},
  C.\ R.\ Math.\ Acad.\ Sci.\ Paris \textbf{334} (2002), no.\ 6, 463--468.

\bibitem[Cas06]{toric}
C.\ Casagrande, \emph{The number of vertices of a Fano polytope}, 
Ann.\ Inst.\ Fourier (Grenoble) 
\textbf{56} (2006), no.\ 1, 121--130.

\bibitem[Cas08]{quasi}
C.\ Casagrande, \emph{Quasi-elementary contractions of Fano manifolds}, 
Compos.\ Math.\ 
\textbf{144} (2008), no.\ 6, 1429--1460.

\bibitem[Cas09]{divcurve}
C.\ Casagrande, \emph{On Fano manifolds with a birational contraction sending a
  divisor to a curve}, Michigan Math.\ J.\ \textbf{58} (2009), no.\ 3, 783--805.

\bibitem[CMSB02]{CMSB}
K.\ Cho, Y.\ Miyaoka and N.\ I.\ Shepherd-Barron, \emph{Characterizations of 
projective spaces and applications to complex symplectic manifolds}, 
Higher dimensional birational geometry (Kyoto, 1997), 1--88, 
Adv.\ Stud.\ Pure Math.\ \textbf{35}, Math.\ Soc.\ Japan, Tokyo, 2002. 

\bibitem[Har77]{har}
R.\ Hartshorne, \emph{Algebraic geometry}, 
Graduate Texts in Mathematics, No.\ 52.\ Springer-Verlag, New York-Heigelberg, 1977. 

\bibitem[Keb02]{kebekus}
S.\ Kebekus, \emph{Families of singular rational curves}, 
J.\ Algebraic Geom.\ \textbf{11} (2002), no.\ 2, 245--256 

\bibitem[KM98]{KoMo}
J.\ Koll{\'a}r and S.\ Mori, \emph{Birational geometry of algebraic varieties},
Cambridge Tracts in Math, vol.\ 134,
Cambridge University Press, Cambridge, 1998.

\bibitem[Kol96]{kollar}
J.\ Koll{\'a}r, \emph{Rational curves on algebraic varieties}, Ergebnisse der
  Mathematik und ihrer Grenzgebiete, vol.\ 32, Springer-Verlag, 1996.

\bibitem[MM81]{MoMu}
S.\ Mori and S.\ Mukai, \emph{Classification of Fano {$3$}-folds with 
$B_2\geq 2$}, 
Manuscripta Math.\ \textbf{36} (1981), no.\ 2, 147--162. 
Erratum: {\bf 110} (2003), no.\ 3, 407.

\bibitem[Mor82]{Mor}
S.\ Mori, \emph{Threefolds whose canonical bundles are not numerically effective},
Ann.\ of Math.\ (2)  \textbf{116}  (1982), no.\ 1, 133--176.

\bibitem[Muk88]{mukai}
S.\ Mukai, \emph{Problems on characterization of the complex projective space},
  Birational Geometry of Algebraic Varieties, Open Problems, Proceedings of the
  23rd Symposium of the Taniguchi Foundation at Katata, Japan, 1988,
  pp.57--60.

\bibitem[NO10]{NoOc}
C.\ Novelli and G.\ Occhetta, \emph{Rational curves and bounds on the Picard
  number of Fano manifolds}, Geom.\ Dedicata \textbf{147} (2010), 207--217.

\bibitem[Occ06]{Occ}
G.\ Occhetta, \emph{A characterization of products of projective spaces}, 
Canad.\ Math.\ Bull.\ \textbf{49} (2006), no.\ 2, 270--280. 

\bibitem[Tsu10a]{tsu1}
T.\ Tsukioka, \emph{Fano manifolds obtained by blowing up along curves with maximal 
Picard number}, Manuscripta Math.\ \textbf{132} (2010), no.\ 1--2, 247--255.

\bibitem[Tsu10b]{tsu2}
T.\ Tsukioka, \emph{A remark on Fano $4$-folds having $(3,1)$-type extremal contractions},
Math.\ Ann.\ \textbf{348} (2010), no.\ 3, 737--747.

\bibitem[Tsu10c]{tsuproc}
T.\ Tsukioka, \emph{Pseudo-index and the length of extremal rays of Fano manifolds} 
(in Japanese), 
WebProceedings of Mini-Conference on Algebraic Geometry in Saitama University, 
March 1--2, 2010. See
\texttt{http://www.rimath.saitama-u.ac.jp/lab.jp/fsakai/proc2010e.html/}

\bibitem[Tsu12]{tsu3}
T.\ Tsukioka, \emph{On the minimal length of extremal rays for Fano four-folds}, 
Math.\ Z.\ \textbf{271} (2012), no.\ 1--2, 555-564. 

\bibitem[Wi{\'s}90]{wisn2}
J.\ A.\ Wi{\'s}niewski, \emph{On a conjecture of Mukai}, Manuscripta Math.\
  \textbf{68} (1990), no.\ 2, 135--141.

\bibitem[Wi{\'s}91]{wisn}
J.\ A.\ Wi{\'s}niewski, \emph{On contractions of extremal rays of Fano manifolds}, J.\ Reine Angew.\ Math.\ 
\textbf{417} (1991), 141--157.

\end{thebibliography}
\end{document}